\documentclass[10pt, reqno]{amsart}
\usepackage{amsmath}
\usepackage{amscd}
\usepackage{amsthm}
\usepackage{amssymb} \usepackage{latexsym}
\usepackage{eufrak}
\usepackage{euscript}
\usepackage{epsfig}
\usepackage{graphics}
\usepackage{array}
\usepackage{enumerate}
\usepackage{dsfont}
\usepackage{color}
\usepackage{wasysym}

\newcommand{\bel}[1]{\begin{equation}\label{#1}}

\newcommand{\be}{\begin{equation}}

\newcommand{\ba}{\begin{eqnarray}}
\newcommand{\ea}{\end{eqnarray}}

\newcommand{\qe}{\end{equation}}
\newcommand{\R}{{\mathbb R}}
\newcommand{\N}{{\mathbb N}}
\newcommand{\Z}{{\mathbb Z}}

\newcommand{\loc}{{\mathrm{loc}}}

\newcommand{\td}{\widetilde}
\newcommand{\<}{\langle}

\theoremstyle{theorem}
\newtheorem{theorem}{Theorem}[section]
\newtheorem{proposition}{Proposition}[section]
\theoremstyle{example}

\theoremstyle{corollary}
\newtheorem{corollary}{Corollary}[section]
\theoremstyle{lemma}
\newtheorem{lemma}{Lemma}[section]
\theoremstyle{definition}

\theoremstyle{proof}

\theoremstyle{remark}
\newtheorem{remark}{Remark}[section]
\theoremstyle{remark}

\begin{document}

\title{Liouville theorems for $f$-harmonic maps into Hadamard spaces}

\author{Bobo Hua}
\address{
School of Mathematical Sciences, LMNS, Fudan University, Shanghai 200433, China; Shanghai Center for Mathematical Sciences, Fudan University, Shanghai 200433, China}\email{bobohua@fudan.edu.cn}
\author{Shiping Liu}
\address{Department of Mathematical Sciences, Durham University, DH1 3LE Durham,
United Kingdom\\
Current address: School of Mathematical Sciences, University of Science and Technology of China, Hefei 230026, China}
\email{spliu@ustc.edu.cn}
\author{Chao Xia}
\address{School of Mathematical Sciences, Xiamen University, 361005, Xiamen, China}
\email{chaoxia@xmu.edu.cn}

\thanks{The research leading to these
results has received funding from the European Research Council
under the European Union's Seventh Framework Programme
(FP7/2007-2013) / ERC grant agreement n$^\circ$~267087. B.H. is supported in part by NSFC Grant no. 11401106. S.L. was  supported in part by the EPSRC Grant EP/K016687/1. C.X. is  supported in part by NSFC Grant No. 11501480.}

\begin{abstract}
In this paper,  we study harmonic functions on weighted manifolds
and harmonic maps from weighted manifolds into Hadamard spaces
introduced by Korevaar and Schoen. We prove several Liouville type
theorems for these harmonic maps.

\end{abstract}
\maketitle

\section{Introduction}
Weighted Riemannian manifolds, also called manifolds with density
 or smooth metric measure spaces
in the literatures, are Riemannian manifolds equipped with weighted
measures. Appearing naturally in the study of self-shrinkers, Ricci
solitons, harmonic heat flows and many others, weighted
manifolds are proved to be nontrivial generalizations of Riemannian
manifolds. There are many geometric investigations of weighted
manifolds, see Morgan \cite{Morgan05}, Wei-Wylie
\cite{WeiWylie09} and many others. In this paper, we investigate various Liouville
type theorems for harmonic functions on weighted manifolds as well
as harmonic maps from weighted manifolds into Hadamard spaces, i.e.
globally nonpositively curved spaces in the sense of Alexandrov (also
called $\mathrm{CAT}(0)$ spaces), see e.g. \cite{Jost97NPC, BuragoBuragoIvanov01}.

 A weighted Riemannian manifold is a triple $(M, g, e^{-f}dV_g)$, where $(M, g)$
is an $n$-dimensional Riemannian manifold, $dV_g$ is the Riemannian
volume element induced by the metric $g$ and $f$ is a smooth positive function on $M$. 
The $f$-Laplacian
$$\Delta_f=\Delta-\nabla f\cdot\nabla$$ is a natural generalization
of Laplace-Beltrami operator $\Delta$ as it is self-adjoint with respect to
the weighted measure $e^{-f}dV_g$, i.e.
$$\int_M \Delta_f u v e^{-f}dV_g=\int_M u \Delta_f v e^{-f}dV_g \hbox{ for } u,v \in C^\infty_0(M).$$ A function $u\in W^{1,2}_{\loc}(M)$ is called $f$-harmonic ($f$-subharmonic, $f$-superharmonic resp.) if it satisfies $\Delta_f u= 0\ (\geq 0, \leq 0\ \mathrm{resp.}) $ in the weak sense, i.e.
$$\int_M \<\nabla u,\nabla \varphi\rangle e^{-f}dV_g= 0\ (\leq 0, \geq 0 \ \mathrm{resp.}) \hbox{ for\ any\ } 0\leq \varphi\in C^\infty_0(M).$$
The Dirichlet $f$-energy of $u$ is defined by $$D^f(u)=\int_M |\nabla
u|^2 e^{-f}dV_g.$$

 On the other hand, $f$-harmonic maps from weighted manifolds $(M, g, e^{-f}dV_g)$ to (singular) metric spaces $(Y,d)$ have wide geometric applications.
 Harmonic maps into metric spaces were initiated by Gromov-Schoen \cite{GromovSchoen92} and then investigated independently by Korevaar-Schoen \cite{KorevaarSchoen93} and Jost \cite{Jost94}. In particular,
 when the domain is a Riemannian manifold, Korevaar-Schoen gave a complete exposition in \cite{KorevaarSchoen93,KorevaarSchoen97}. In this paper we call
a map $u: M\to Y$
 $f$-harmonic if $u$ locally minimizes the $f$-energy functional $E^f$ in the sense of Korevaar-Schoen. For a detailed definition and its properties, we refer to \cite{KorevaarSchoen93} or Section 4 below.
For the special case, $f$-harmonic
 maps from the Gaussian spaces, $(\R^n, |\cdot|, e^{-|x|^2/4}dx)$, to Riemannian manifolds are called quasi-harmonic spheres which emerge in the blow-up analysis of harmonic heat flow \cite{LinWang99,LiTian00}.
In this paper, we study Liouville theorems for $f$-harmonic maps
into metric spaces, which generalize the previous results for
harmonic maps in both aspects of domain manifolds and target spaces.

Analysis on weighted manifolds and  the corresponding $f$-Laplacian have
been extensively studied recently. We refer to
\cite{MunteanuWang11,MunteanuWang12,Brighton13,Lixiangdong05,Lixiangdong13}
for the $f$-harmonic functions on weighted manifolds, to
\cite{LiWang09,ZhuWang10,LiZhu10,LiYang12} for $f$-harmonic
functions on the Gaussian spaces, to \cite{Grigoryan06,Grigoryan09}
for heat kernel estimates, and to
\cite{LinWang99,WX,ChenJostQiu12,RV,Sinaei12a,Sinaei12b} for
$f$-harmonic maps.

\

In the first part of the paper we are concerned with Liouville type theorems
for $f$-harmonic functions on weighted manifolds. Several Liouville type theorems for $f$-harmonic functions on the Gaussian
spaces, also called quasi-harmonic functions, have been proved in
\cite{ZhuWang10,LiWang09}, in which main techniques adopted are gradient
estimates and separation of variables coupled with ODE results. In
this paper, we propose another approach, which seems to be
overlooked in the literatures, to reprove many previous results. This method can be easily generalized, so that we may obtain Liouville
theorems for $f$-harmonic functions for a large class of weighted manifolds, see Section~\ref{section2}.

Our observation is that the weighted version of $L^p$-Liouville theorem for weighted manifolds can be used to derive various
Liouville theorems concerning the growth of $f$-harmonic functions.
Yau \cite{Yau76} first proved the $L^{p}$-Liouville theorem
($1<p<\infty$) for harmonic functions on any complete Riemannian manifold.
Later, Karp \cite{Karp82} obtained a quantitative version of this
result. Li-Schoen \cite{LiSchoen84} proved other $L^p$-Liouville
theorems (e.g. $0<p<1$) under the curvature assumption of manifolds.
Karp's version of $L^p$-Liouville theorem has been generalized by
Sturm \cite{Sturm94} to the setting of strongly local regular
Dirichlet forms. In particular, our $f$-harmonic functions lie in
this setting. By applying Sturm's $L^p$-Liouville theorem to
$f$-harmonic functions, we immediately obtain several consequences
which generalize previous results of
\cite{ZhuWang10,LiWang09,LiZhu10,LiYang12}. Although the proof of
$L^p$-Liouville theorem is quite general which only involves the
integration by parts and the Caccioppoli inequality (thus it holds
for all reasonable spaces), it is surprisingly powerful to obtain
various Liouville theorems for weighted
manifolds with slow volume growth, especially for the Gaussian
spaces, see Corollaries~\ref{Corollary1} and \ref{gauss} in
Section~\ref{section2}. This does provide another approach to derive
Liouville theorems without using any gradient estimate.

\

The second part of this paper is devoted to the study of Liouville
type theorems for harmonic maps from weighted manifolds to Hadamard
spaces. For applications of $f$-harmonic maps with singular targets
we refer to Gromov-Schoen \cite{GromovSchoen92}. Our first result is
an analogue to Kendall's theorem \cite[Theorem~3.2]{Kendall90}. The
essence of Kendall's theorem is that validity of Liouville theorem
for $f$-harmonic maps into Hadamard spaces, a priori a nonlinear
problem, is reduced to that of Liouville theorem of $f$-harmonic
functions, a linear problem. 
Kendall \cite{Kendall90} proved this theorem for harmonic maps
between Riemannian manifolds, by using probabilistic methods and
potential theory. Kuwae-Sturm \cite{KuwaeSturm08} generalized
Kendall's method to a class of harmonic maps between general metric
spaces in the framework of Markov processes.
Note that harmonic maps they were concerned with are different from those of Korevaar-Schoen \cite{KorevaarSchoen93} when targets are singular. 
In this paper, we consider harmonic maps into Hadamard spaces in the sense of Korevaar-Schoen. Following the argument by Li-Wang
\cite{LiWang98}, we are able to prove the following Kendall type theorem by assuming local compactness of the targets.
Recall that a geodesic space $(Y,d)$ is called locally compact if
every closed geodesic ball is compact.
\begin{theorem}\label{Kendall}
Let $(M, g, e^{-f}dV_g)$ be a complete weighted Riemannian manifold satisfying that any bounded $f$-harmonic function is constant. Let $(Y,d)$
be a locally compact Hadamard space. Then  any $f$-harmonic map from
$M$ to $Y$ having bounded image is a constant map.
\end{theorem}

\

In the same spirit of Kendall's theorem, Cheng-Tam-Wan \cite{CTW}
proved a Liouville type theorem for harmonic maps with finite
energy. Our second result is a generalization of Cheng-Tam-Wan's
theorem to $f$-harmonic maps into Hadamard spaces.
\begin{theorem}\label{Cheng} Let $(M, g, e^{-f}dV_g)$ be a complete noncompact weighted Riemannian manifold satisfying that any $f$-harmonic function with finite Dirichlet $f$-energy is bounded. Let $(Y,d)$ be an Hadamard space. Then any $f$-harmonic map from $M$ to $Y$ with finite $f$-energy has bounded image.
\end{theorem}
We will follow the line of Cheng-Tam-Wan's reasoning, but using the
techniques in potential theory, especially the theory of
Royden-Nakai's decomposition on Riemannian
manifolds \cite{Royden52,Nakai60,SarioNakai70}. This possible
approach of potential theory was implicitly suggested by Lyons in
\cite[pp.~278]{CTW}. We figure out the detailed arguments of this
insight and apply them to Liouville theorems of $f$-harmonic maps.
Royden-Nakai's decomposition theorem and Virtanen's theorem, see e.g. Section~\ref{section5} for weighted versions, play important roles in the classification theory of Riemannian manifolds
developed by Royden, Nakai, Sario et al. many years ago.
We shall dwell on these theories in the framework of weighted
manifolds in Section 5 and utilize them to prove Theorem \ref{Cheng}.

The following theorem is, more or less, a consequence of the
combination of Theorem \ref{Kendall} and \ref{Cheng}.
\begin{theorem}\label{main thm}
Let $(M,g, e^{-f}dV_g)$ be a complete noncompact weighted Riemannian
manifold satisfying that any bounded $f$-harmonic functions is constant. Let $(Y,d)$ be a locally compact Hadamard space. Then any  $f$-harmonic
map from $M$ to $Y$ with
finite $f$-energy 
is a constant map.
\end{theorem}

This theorem has an interesting application which motivates our studies in some sense.
Bakry-\'Emery \cite{BakryEmery85} introduced weighted Ricci curvature for weighted manifolds. 
In particular, the so-called $\infty$-Bakry-\'Emery Ricci curvature
\begin{eqnarray*}
Ric_f:=Ric+\nabla^2 f
\end{eqnarray*} turns out to be a suitable and important curvature quantity
for weighted manifolds. The nonnegativity of $Ric_f$ corresponds to
the curvature-dimension condition $CD(0,\infty)$ on metric measure
spaces via optimal transport, in the sense of Lott-Villani
\cite{LottVillani09} and Sturm \cite{Sturm06a,Sturm06b}.
By a theorem of Brighton \cite{Brighton13}, see also Li
\cite{Lixiangdong13}, the weighted manifold $(M,g, e^{-f}dV_g)$
satisfying $Ric_f\geq 0$ admits no nonconstant bounded $f$-harmonic
functions. Hence by Theorem \ref{main thm} we immediately have
\begin{theorem}\label{main thm1}
Let $(M,g, e^{-f}dV_g)$ be a complete noncompact weighted Riemannian
manifold satisfying $Ric_f\geq 0$ and $(Y,d)$ be a locally compact
Hadamard space. Then any  $f$-harmonic map from $M$ to $Y$ with
finite $f$-energy
is a constant map.
\end{theorem}

The novelty of the result lies in the generality of targets, i.e., including singular metric spaces. In the smooth setting, Hadamard spaces are in fact Cartan-Hadamard manifolds, i.e. simply connected Riemannian manifolds with nonpositive sectional curvature. On Riemannian manifolds, Theorem~\ref{main thm1} has been proved by Wang-Xu \cite{WX} and
Rimoldi-Veronelli \cite{RV} independently under an additional assumption of $\int_M
e^{-f}dV_g=\infty$ for domain manifolds, while simply connectedness of the targets is not needed. Note that the weighted volume assumption here cannot be derived from the curvature condition $Ric_f\geq
0$ in
general. In addition, there is a nontrivial $f$-harmonic map from a domain manifold with $Ric_f\geq 0$ and $\int_M e^{-f}dV_g<\infty$ to a nonpositively curved target manifold, constructed by \cite[Remark~3.7]{RV}.  Our contribution is to drop the weighted volume assumption by assuming simply connectedness of the targets and to extend the result to singular spaces.

For harmonic maps into singular Hadamard spaces, the arguments in
\cite{WX,RV}, both following Schoen-Yau
\cite{SchoenYau76}, do not work any more since we cannot apply
Bochner techniques as in \cite{WX,RV} due to the
singularity of targets. Although a weak Bochner formula can also be derived following Korevaar-Schoen \cite{KorevaarSchoen93}, it is insufficient for our purpose. 
Fortunately, we can circumvent these technical problems by 
 proving Theorem \ref{main thm} which follows from Kendall type theorems. This does provide another approach to Liouville theorems for $f$-harmonic maps without using Bochner techniques.
 This is one of the main points of the paper.

\

The rest of the paper is organized as follows. In Section 2, we
study $L^p$ Liouville theorem for $f$-harmonic functions and give
some applications. In Section 3, we consider harmonic maps with smooth targets. In Section 4, we
define $f$-harmonic maps into
Hadamard spaces and
prove Theorem \ref{Kendall}. 
In Section 5, we dwell on the Royden-Nakai theory and prove Theorems \ref{Cheng} and 
\ref{main thm}.

\section{$f$-harmonic functions}\label{section2}

In this section, we study $L^p$-Liouville theorem for
$f$-harmonic functions and its applications. We will show that $L^p$-Liouville theorem is quite powerful for weighted manifolds with
finite volume.

The $L^p$-Liouville theorem, $1<p<\infty$, for harmonic functions
(or nonnegative subharmonic functions) was initiated by Yau
\cite{Yau76} on complete Riemannian manifolds. Karp \cite{Karp82}
obtained a quantitative version of this Liouville theorem. Later,
Sturm \cite{Sturm94} proved $L^p$-Liouville theorem for strongly
local regular Dirichlet forms. The following theorem is a special
case of Sturm's result for $f$-harmonic functions. We denote by
$B_r:=B_r(x_0)$ the closed geodesic ball of
radius $r$ centered at a fixed point $x_0\in M$. 

\begin{theorem}[{\cite[Theorem~1]{Sturm94}}]\label{quantitativethem}
Let $(M, g ,e^{-f}dV_g)$ be a complete weighted Riemannian manifold
and $u$ be a nonnegative $f$-subharmonic function (or an
$f$-harmonic function). For $1<p<\infty,$ set
$v(r):=\int_{B_r}|u|^pe^{-f}dV_g.$ Then either
\begin{equation*}\label{quantitative}
\inf_{a>0}\int_{a}^{\infty}\frac{r}{v(r)}dr<\infty,
\end{equation*} or  $u$ is a constant.
\end{theorem}

We state several consequences of Theorem \ref{quantitativethem}.

A quite useful consequence is about $f$-parabolicity of $M$. Recall
that a weighted manifold $(M, g ,e^{-f}dV_g)$ is called
\emph{$f$-parabolic} if there is no nonconstant nonnegative
$f$-superharmonic functions on $M$. For a compact set $K\subset M,$
the \emph{$f$-capacity} of $K$ is defined as
$$\mathrm{cap}^f(K):=\inf_{\substack{\phi\in \mathrm{Lip}_0(M)\\ \phi|_K=1}}\int_M|\nabla
\phi|^2e^{-f}dV_g,$$ where $\mbox{Lip}_0(M)$ is the space of
compactly supported Lipschitz functions on $M.$
\begin{proposition}[$f$-parabolicity]\label{f-parabolic}
Let $(M, g ,e^{-f}dV_g)$ be a complete weighted manifold. Then the
following are equivalent:
\begin{enumerate}
\item[(i)]  $M$ is $f$-parabolic;
\item[(ii)]  $\mathrm{cap}^f(K)=0$ for some (then any) compact set $K\subset M$;
\item[(iii)]  any bounded $f$-superharmonic function on $M$ is constant.
\end{enumerate}
\end{proposition}
\begin{proof}
$(i)\Leftrightarrow(ii)$. This follows from
\cite[Proposition~3]{Grigoryan87}, see
\cite[Proposition~2.1]{Grigoryan99}.

$(i)\Leftrightarrow(iii)$. This follows from the fact that any
nonnegative $f$-superharmonic function $u$ can be approximated by
bounded $f$-superharmonic functions $u_n=\min\{u,n\},$ $n\in\N.$
\end{proof}

We say a weighted manifold $(M, g ,e^{-f}dV_g)$ has the
\emph{moderate volume growth property} if
\begin{eqnarray}\label{moderate}
\int_{1}^{\infty}\frac{r}{V_f(B_r)}dr=\infty,
\end{eqnarray}
where  $V_f(B_r):=\int_{B_r}e^{-f}dV_g$.
\begin{corollary}\label{fparabolic}
Let $(M, g ,e^{-f}dV_g)$ be a complete weighted Riemannian manifold
satisfying the moderate volume growth property. Then $M$ is
$f$-parabolic.
\end{corollary}

\begin{proof}
Let $u$ be a bounded $f$-superharmonic function on $M.$ Then for any
$a>0,$
$$\int_{a}^{\infty}\frac{r}{v(r)}dr\geq
C\int_{a}^{\infty}\frac{r}{V_f(B_r)}dr=\infty.$$ Theorem
\ref{quantitativethem} yields that $u$ is a constant. This proves
the corollary.
\end{proof}
\begin{remark}
Corollary \ref{fparabolic} slightly generalizes
\cite[Theorem~1.4]{WX}. In particular, this corollary implies
\cite[Theorem 2]{ZhuWang10}.
\end{remark}

We can also derive several Liouville type theorems for $f$-harmonic
functions from Theorem \ref{quantitativethem}.
\begin{corollary}\label{Corollary1}
Let $(M, g ,e^{-f}dV_g)$ be a complete weighted Riemannian manifold
and  $u$ be a nonnegative $f$-subharmonic function (or $f$-harmonic
function). Assume one of the following holds:
\begin{enumerate}
\item[(i)] $u=O(w^\alpha)$ for some nonnegative function $w$ satisfying $\int_M wd^{-2}(\cdot,x_0)e^{-f}dV_g<\infty$ and some $\alpha\in
(0,1)$;
\item[(ii)] $\int_M d^k(\cdot,x_0)e^{-f}dV_g<\infty$ for some $k>-2$ and $u=O(d^\beta(\cdot,x_0))$ for some $\beta\in (0,
k+2)$;
\item[(iii)] $\int_M e^{-f}dV_g<\infty$ and $u=O(d^{\beta}(\cdot,x_0))$ for $\beta\in(0,2)$;
\item[(iv)] $f\geq Cd(\cdot,x_0)^{\beta}$ for some $C>0,\beta>0$ and
$\int_M e^{-\delta f}dV_g<\infty$ for some $0<\delta<1$ and  $u$ has
polynomial growth;
\item [(v)] $f\geq Cd(\cdot,x_0)^{\beta}$ for some $C>0,\beta>0$ and
the Riemannian volume has polynomial volume growth and  $u=O(e^{\alpha Cd(\cdot,
x_0)^{\beta}})$, $\alpha\in (0,1)$.
\end{enumerate}
 Then $u$ is a constant.
\end{corollary}

\begin{proof} For (i), we see that there exists $p\in (1,\infty)$ such that
$|u|^p=O(w).$ Hence
\begin{eqnarray*}\frac{1}{r^2\log r}v(r)&=&\frac{1}{r^2\log
r}\int_{B_r}|u|^pe^{-f}dV_g\\ &\leq& \frac{C}{\log
r}\int_{B_r}\frac{w(x)}{d^2(x,x_0)}e^{-f(x)}dV_g(x)=o(1).
\end{eqnarray*}
It follows from Theorem \ref{quantitativethem} that $u$ is a
constant. (ii) follows from (i) by letting $w=d^{k+2}(\cdot, x_0).$
(iii) follows from (ii) by letting $k=0$.

For (iv), let us observe for any $1<p<\infty,$
$$\int_M |u|^pe^{-f}dV_g\leq C\int_{M}d^{sp}(x,x_0)e^{-f(x)}dV_g(x)\leq C\int_M e^{-\delta
f}dV_g<\infty,$$ where $s>0.$ Then the statement also follows from
Theorem \ref{quantitativethem}. (v) can be proved in a similar way.
\end{proof}

The following result is a direct corollary of the above (v).
\begin{corollary}\label{gauss}
Let $u$ be an $f$-harmonic function on the Gaussian space, i.e.,
$$\Delta u-\frac12\<x,\nabla u\rangle=0.$$ If
$u=O(e^{\alpha\frac{|x|^2}{4}})\hbox{ as } x\to\infty,$ for some
$0<\alpha<1,$ then $u$ is a constant.
\end{corollary}

\begin{remark}
Corollary \ref{gauss} implies that there is no nonconstant
polynomial growth $f$-harmonic functions on the Gaussian space. This
improves the result in \cite[Theorem~4.2]{LiWang09}. By
Caccioppoli's inequality, Corollary \ref{gauss} can be also derived
from Li-Yang \cite[Corollary~1.2]{LiYang12} .
\end{remark}

In the remaining part of this section, we study the $L^p$-Liouville
theorem introduced by Zhu-Wang \cite{ZhuWang10} using a different
measure from ours. We shall explain why the critical exponent of
their $L^p$-Liouville theorem in \cite[Theorem~3]{ZhuWang10} is
$p=\frac{n}{n-2}$ ($n\geq 3$) by applying our result. Let
$(M,g,e^{-f}dV_{g})$ be an $n$-dimensional ($n\geq 3$) complete
weighted manifold. In fact, they consider the $L^p$ space with
respect to the
 Riemannian volume
in a modified Riemannian manifold $\widetilde{M}=(M,
\tilde{g},dV_{\tilde{g}})$, denoted by $L^p(\td{M}, dV_{\tilde{g}}),$ where $\tilde{g}$ is a conformal change of $g$ given by $\tilde{g}=e^{-\frac{2f}{n-2}}g$. Since this new manifold $\td{M}$
may be incomplete, e.g. for the Gaussian space, Yau's $L^p$-Liouville theorem fails in this setting. In the following, we use
the $L^p$-Liouville theorem on weighted manifolds to show the one on
the modified Riemannian manifolds.

\begin{theorem}\label{modified}
Let $(M,g,e^{-f}dV_{g})$ be an $n$-dimensional ($n\geq 3$) complete
weighted manifold, $\widetilde{M}=(M, \tilde{g},dV_{\tilde{g}})$ be
the modified Riemannian manifold and $u$ be a nonnegative
$f$-subharmonic function (or $f$-harmonic function) on $M$. For any
$p>\frac{n}{n-2},$ there exists a constant $\delta=\delta(p,n)\in
(0,1)$ such that if $\int_Me^{-\delta f }dV_g<\infty$ and $u\in
L^p(\td{M},dV_{\tilde{g}}),$ then $u$ is a constant.
\end{theorem}
\begin{proof}
For any $p>\frac{n}{n-2},$ let $q=\frac{2p}{p+\frac{n}{n-2}}>1$,
$\alpha=\frac{p}{q}>\frac{n}{n-2}$ and
$\alpha^*=\frac{\alpha}{\alpha-1}\in (1,\frac{n}{2})$. Set  $\delta
=\frac{n-2\alpha^*}{n-2}\in (0,1).$ By H\"older inequality, we can
verify that
\begin{eqnarray*}
\int_{M}u^qe^{-f}dV_{g}=\int_{M}u^qe^{\frac{2f}{n-2}}dV_{\tilde{g}}&\leq
&
\left(\int_{M}u^{q\alpha}dV_{\tilde{g}}\right)^{\frac{1}{\alpha}}\left(\int_{M}e^{\frac{2\alpha^*f}{n-2}}dV_{\tilde{g}}\right)^{\frac{1}{\alpha^*}}\\&=&\left(\int_M
u^p dV_{\tilde{g}}\right)^{\frac{1}{\alpha}}\left(\int_M e^{-\delta
f}dV_{g}\right)^{\frac{1}{\alpha^*}}<\infty.
\end{eqnarray*}
The statement follows from Theorem \ref{quantitativethem}.
\end{proof}
This yields a direct corollary which generalizes
\cite[Theorem~3]{ZhuWang10}, which is restricted to the Gaussian
spaces, to general weighted manifolds. 
The Riemannian manifold
$(M,g,dV_g)$ is called of sub-exponential volume growth if
$V_g(r):=V_g(B_r(x_0))=e^{o(r)}$ for some (then all) $x_0\in M.$
\begin{corollary}
Let $(M,g,e^{-f}dV_{g})$ be an $n$-dimensional ($n\geq 3$) complete
weighted manifold satisfying that $f\geq Cd^{\beta}(\cdot,x_0)$ for
some $C>0,\beta>0$ and $V_g(r)=e^{o(r^\beta)}.$
Let $\widetilde{M}=(M, \tilde{g},dV_{\tilde{g}})$ be the modified
Riemannian manifold. Then for any $p>\frac{n}{n-2},$ the
$f$-harmonic function in $L^p(\td{M},dV_{\tilde{g}})$ is constant.
In particular, for $\beta=1,$ it suffices to assume $(M,g,dV_g)$
has subexponential volume growth.
\end{corollary}
\begin{proof}
By virtue of Theorem \ref{modified}, it is sufficient to prove $\int_M e^{-\delta f}dV_g<\infty$ where $\delta$ is the constant in Theorem \ref{modified}.
We see by the co-area formula that
\begin{eqnarray*}
\int_M e^{-\delta f}dV_g&=&\int_0^1  \int_{S_r(x_0)} e^{-\delta f}dA_rdr+ \int_1^\infty \int_{S_r(x_0)} e^{-\delta f}dA_rdr\\&\leq &C_0+ \int_1^\infty \int_{S_r(x_0)} e^{-\delta Cr^\beta}\\
&=&C_0+ \int_1^\infty e^{-\delta Cr^\beta} \frac{d}{dr}V_g(r) dr\\&=&C_0+e^{-\delta Cr^\beta}V_g(r)\big|_1^\infty+\delta C\int_1^\infty \beta r^{\beta-1}e^{-\delta Cr^\beta} V_g(r) dr.
\end{eqnarray*}
Since $V_g(r)=e^{o(r^\beta)}$, there exists $R$ large such that   $$V_g(r)\leq e^{\frac12\delta Cr^{\beta}}\hbox{ for }r>R.$$
It follows that $\lim_{r\to\infty} e^{-\delta Cr^\beta}V_g(r)=0$ and $\int_1^\infty \beta r^{\beta-1}e^{-\delta Cr^\beta} V_g(r) dr<\infty$.
It follows that $\int_M e^{-\delta f}dV_g<\infty$. We complete the proof.
\end{proof}

\

\section{$f$-harmonic maps into Cartan-Hadamard manifolds}

In this section, we prove Theorem \ref{main thm1} in the case that
the target $Y=N$ is a Cartan-Hadamard manifold. 

\begin{theorem}\label{thm1}
Let $(M, g, e^{-f}dV_g)$ be a complete weighted Riemannian manifold
which is $f$-parabolic
 and $N$  be a Cartan-Hadamard manifold. Then any
$f$-harmonic map  with finite $f$-energy,  i.e. $E^f(u):=\int_M
|\nabla u|^2e^{-f}dV_g<\infty$, is a constant map.
\end{theorem}

\begin{proof} We use a construction by Rimoldi-Veronelli \cite{RV} which associates an $f$-harmonic map with a  harmonic map on some higher dimensional warped product manifold.

Precisely,  let $\bar{M}:= M \times_{e^{-f}} \mathbb{S}^1$ denote a
warped product, where $\mathbb{S}^1=\R/\Z$ with
$Vol(\mathbb{S}^1)=1$, with the metric on $\bar{M}$ given by
$\bar{g}(x, t) = g(x)+e^{-2f(x)}dt^2$. Note that $\bar{M}$ is
complete. It follows from \cite[Proposition 2.5 and  Lemma 2.6]{RV}
that $\bar{M}$ is parabolic and   the map $\bar{u}:\bar{M}\to N$,
defined by $\bar{u}(x,t)=u(x)$ is a harmonic map. Moreover,
$E_{\bar{M}}(\bar{u})=E_M^f(u)<\infty$.

Now by applying \cite[Proposition~2.1~and~Theorem 3.1]{CTW} to
$\bar{u}$ and $\bar{M}$, we know that the image of $\bar{u},$
$\bar{u}(\bar{M})=u(M)$, is bounded in $N.$ Since $N$ is a
Cartan-Hadamard manifold, $d^2(\bar{u}(\cdot), Q)$ is a subharmonic
function for any $Q\in N$, which is also bounded. By the
parabolicity of $\bar{M},$ we know that $d^2(\bar{u}(\cdot), Q)$ is
constant for any $Q\in N.$ This proves the theorem.
\end{proof}


\begin{theorem}\label{thm2}
Let $(M, g, e^{-f}dV_g)$ be a complete weighted Riemannian manifold
satisfying $Ric_f\geq 0$
 and $N$  be a Cartan-Hadamard manifold. Then any
$f$-harmonic map  with finite $f$-energy  $E^f(u)<\infty$ is a
constant map.
\end{theorem}

\begin{proof}
We divide the theorem into two cases: (a) $\int_M
e^{-f}dV_g=\infty$, (b) $\int_M e^{-f}dV_g<\infty$. For the case
(a), it was already proved in \cite[Theorem~1.2]{WX} or
\cite[Theorem~3.3]{RV} for general Riemannian target of nonpositive
curvature (without the assumption of simply connectedness). For the
case (b), we observe that $M$ satisfies the moderate volume growth
property \eqref{moderate}. By Corollary \ref{fparabolic}, $M$ is
$f$-parabolic. Then the statement follows from Theorem \ref{thm1}.
Hence we finish the proof.
\end{proof}

\begin{remark}
Comparing Theorem \ref{thm2} with \cite[Theorem~1.2]{WX} or
\cite[Theorem~3.3]{RV},  we remove the condition of the infinity of
$f$-volume for $M$ but add the assumption that $N$ is simply
connected. 
\end{remark}

\

\section{$f$-harmonic maps into Hadamard spaces}
In this section, we define $f$-harmonic maps from an $n$-dimensional
complete weighted Riemannian manifold $(M, g, e^{-f}dV_g)$ to a
general metric space $(Y, d)$. For that purpose we investigate an
$f$-energy functional $E^f$ whose definition given here follows
Korevaar-Schoen \cite{KorevaarSchoen93}, 
 where a Sobolev space
theory for maps from Riemannian domains to metric spaces was developed.
Note that the energy functional has been further extended to maps
from complete noncompact Riemannian manifolds, even more
generally the so-called admissible Riemannian polyhedrons with
simplexwise smooth Riemannian metric, in Eells-Fuglede \cite{EF}
(see Chapter $9$ therein).

We consider Borel-measurable (equivalently, measurable w.r.t.
$e^{-f}dV_g$) maps $u: M\rightarrow Y$ ($u$ then has separable range
since $M$ is a separable metric space, see \cite[Problem 10 in Section 4.2]{Dudley02}). The space $L^2_{loc}(M_f, Y)$ is
defined as the set of Borel-measurable maps $u$ for which
$d(u(\cdot), Q)\in L^2_{loc}(M, e^{-f}dV_g)$ for some point $Q$ (and
hence for any $Q$ by triangle inequality) in $Y$. Since this space
is unchanged if we use the unweighted measure $dV_g$ instead of
$e^{-f}dV_g$ in its definition, we will write $L^2_{loc}(M, Y)$ for
simplicity in the following. When $M$ is compact, $L^2_{loc}(M, Y)$
is a complete metric space, with distance function $\hat{d}$ defined
by
\begin{equation*}
\hat{d}^2(u, v):=\int_Md^2(u(x), v(x))e^{-f(x)}dV_g(x),
\end{equation*}
provided that $(Y, d)$ is complete.

The approximate energy density for a map $u\in L^2_{loc}(M, Y)$ is
defined for $\varepsilon>0$ as

\begin{equation}
e_{\varepsilon}(u):=\frac{1}{\omega_n}\int_{S(x,\varepsilon)}\frac{d^2(u(x),u(y))}{\varepsilon^2}\frac{d\sigma_{x,\varepsilon}(y)}{\varepsilon^{n-1}},
\end{equation}
where $d\sigma_{x,\varepsilon}(y)$ is the $(n-1)$-dimensional
surface measure on the sphere $S(x,\varepsilon)$ of radius
$\varepsilon$ centered at $x$ induced by the Riemannian metric $g$,
and $\omega_n$ is the volume of the $n$-dimensional unit Euclidean ball.
One can check that the function $e_{\varepsilon}(u)\in L^1_{loc}(M)$
(see \cite{KorevaarSchoen93}). Then we can define the $f$-energy
functional $E^f$ by
\begin{equation*}
E^f(u):=\sup_{\substack{\eta\in C_0(M)\\0\leq\eta\leq
1}}\left(\limsup_{\varepsilon\rightarrow 0}\int_M\eta
e_{\varepsilon}(u)e^{-f}dV_g\right).
\end{equation*}

We say a map $u\in L^2_{loc}(M, Y)$ is locally of finite energy,
denoted by $u\in W_{loc}^{1,2}(M, Y)$, if
$E^f(u_{\mid\Omega})<\infty$ for any relatively compact domain
$\Omega\subset M$.

\begin{theorem}
If $u\in W_{loc}^{1,2}(M, Y)$, then there exists a function $e(u)\in
L^1_{loc}(M)$, such that for any $\eta\in C_0(M)$,
the following limit exists \begin{equation}\label{fHar0}
\lim_{\varepsilon\rightarrow 0}\int_M\eta
e_{\varepsilon}(u)e^{-f}dV_g=:\int_M\eta e(u)e^{-f}dV_g\end{equation} which serves as the definition of $e(u).$
\end{theorem}
\begin{proof}
By definition, $u\in W_{loc}^{1,2}(M, Y)$ implies that for any
connected, open and relatively compact subset $\Omega\subset M$,
$u_{\mid\Omega}\in L^2(\Omega, Y)$ and
\begin{equation*}
\sup_{\substack{\zeta\in C_0(\Omega)\\0\leq\zeta\leq
1}}\left(\limsup_{\varepsilon\rightarrow 0}\int_{\Omega}\zeta
e_{\varepsilon}(u_{\mid\Omega})dV_g\right)<\infty,
\end{equation*}
that is, $u_{\mid \Omega}\in W^{1,2}(\Omega, Y)$ in
Korevaar-Schoen's notation \cite{KorevaarSchoen93}.

Now by \cite[Theorem~1.5.1~and~Theorem~1.10]{KorevaarSchoen93}, we
know that there exists a function $e(u_{\mid\Omega})\in L^1(\Omega)$
such that
\begin{equation}\label{fHar2}
\lim_{\varepsilon\rightarrow
0}\int_{\Omega}\zeta e_{\varepsilon}(u)dV_g=\int_{\Omega}\zeta e(u_{\mid\Omega})dV_g,
\,\,\,\forall\,\,\,\zeta\in C_0(\Omega).
\end{equation}
In particular, one has
\begin{equation}\label{fHar1}
\lim_{\varepsilon\rightarrow 0}\int_{\Omega}\eta
e_{\varepsilon}(u)e^{-f}dV_g=\int_{\Omega}\eta
e(u_{\mid\Omega})e^{-f}dV_g, \,\,\,\forall\,\,\,\eta\in C_0(\Omega).
\end{equation}
We then define a function $e(u)$ on $M$ by
$e(u)_{\mid\Omega}:=e(u_{\mid\Omega})$ for any $\Omega\subset M$
with smooth boundary. One can show that $e(u)$ is well defined. For
that purpose, one only needs to check
$e(u_{\mid\Omega})=e(u_{\mid\Omega_1})$ on $\Omega_1\subset \Omega$
where both $\Omega_1$ and $\Omega\setminus\Omega_1$ have Lipschitz
boundary. This is true since by the trace theory
\cite[Theorem~1.12.3]{KorevaarSchoen93}, one has
\begin{equation*}
\int_{\Omega}e(u_{\mid\Omega})dV_g=\int_{\Omega_1}e(u_{\mid\Omega_1})dV_g+\int_{\Omega\setminus\Omega_1}e(u_{\mid\Omega\setminus\Omega_1})dV_g.
\end{equation*}

Then (\ref{fHar0}) follows from (\ref{fHar1}) which proves this
theorem.
\end{proof}
\begin{remark}
By the definition of $e(u)$ and (\ref{fHar2}), we know
$$e(u)(x)=|\nabla u|^2(x),$$ where $|\nabla u|^2(x)$ is the energy
density function in \cite{KorevaarSchoen93}. This function is
consistent with the usual way of defining $|du|^2$ for maps between
Riemnannian manifolds. Therefore we will use $|\nabla u|^2(x)$
instead of $e(u)(x)$ in the following.
\end{remark}
\begin{remark}\label{fHarrem2}
By a polarization argument, we can check that for any two functions
$h_1,h_2\in W^{1,2}_{loc}(M, e^{-f}dV_g)$,
\begin{align*}
&\lim_{\varepsilon\rightarrow
0}\int_M\eta(x)\frac{1}{\omega_n}\int_{S(x,\varepsilon)}\frac{(h_1(x)-h_1(y))(h_2(x)-h_2(y))}{\varepsilon^2}\frac{d\sigma_{x,\varepsilon}(y)}{\varepsilon^{n-1}}
e^{-f(x)}dV_g(x)\\=&\int_M\eta(x)\<\nabla h_1(x),\nabla
h_2(x)\rangle e^{-f(x)}dV_g(x), \,\,\forall \,\,\eta\in C_0(M).
\end{align*}
\end{remark}
\begin{remark}
With (\ref{fHar0}) in hand, by the definition of $E^f$, we can
derive (see \cite[Theorem~9.1]{EF}),
$$E^f(u)=\int_M |\nabla u|^2 e^{-f}dV_g, \,\,\forall \,\, u\in W_{loc}^{1,2}(M, Y).$$
In particular, we define $D^f(u)=E^f(u)$ when $u$ is a scalar function.
\end{remark}
\begin{remark}
As in \cite{KorevaarSchoen93}, the definition of $E^f$ is unchanged
if we replace $e_{\varepsilon}(x)$ by
${}_{\nu}e_{\varepsilon}(x):=\int_0^2e_{\lambda\varepsilon}(x)d\nu(\lambda)$,
where $\nu$ is any Borel measure on the interval $(0,2)$ satisfying
$\nu\geq 0,\,\,\nu((0,2))=1,\,\,
\int_0^2\lambda^{-2}d\nu(\lambda)<\infty$. For example, the
approximate energy density function can be chosen as follows.
\begin{enumerate}
  \item When $n\geq 3$, for the measure $d\nu_1(\lambda)=n\lambda^{n-1}d\lambda,\,\, 0<\lambda<1$, $${}_{\nu_1}e_{\varepsilon}(x)=\frac{n}{\omega_n}\int_{B(x,\varepsilon)}\frac{d^2(u(x),u(y))}{d^2(x,y)}\frac{dV_g(y)}{\varepsilon^n};$$
  \item For the measure $d\nu_2(\lambda)=(n+2)\lambda^{n+1}d\lambda,\,\, 0<\lambda<1$, $${}_{\nu_2}e_{\varepsilon}(x)=\frac{n+2}{\omega_n}\int_{B(x,\varepsilon)}\frac{d^2(u(x),u(y))}{\varepsilon^2}\frac{dV_g(y)}{\varepsilon^n}.$$
\end{enumerate}
\end{remark}

\begin{remark}
For $n\geq 3,$ by introducing a conformal change of the metric
$\widetilde{M}=(M, \tilde{g},dV_{\tilde{g}})$ where
$\tilde{g}=e^{-\frac{2f}{n-2}}g$ and employing the energy density
${}_{\nu_1}e_{\varepsilon}$, many problems for weighted manifolds
can be reduced to those on (possibly incomplete) unweighted
manifolds. However, we prefer to write the proofs in a unified way
which includes the case $n=2.$
\end{remark}

We call a map $u\in W^{1,2}_{loc}(M, Y)$ \emph{$f$-harmonic} if it
is a local minimizer of the energy functional $E^f$, i.e., for any
connected, open and relatively compact domain $\Omega\subset M$,
$E^{f}(u)\leq E^{f}(v)$ for every map $v\in W^{1,2}_{loc}(M, Y)$
such that $u=v$ in $M\setminus\Omega$.

We now investigate the properties of the function $d(u(\cdot), Q)$
on $M$, where $u: M\rightarrow Y$ is an $f$-harmonic map and $Q\in
Y$. The first observation is that
\begin{equation}\label{fHar3}
E^f(d(u, Q))\leq E^f(u).
\end{equation}
This can be derived from the triangle inequality
$$(d(u(x),Q)-d(u(y), Q))^2\leq d^2(u(x),u(y)).$$

Recall that an \emph{Hadamard space} (also called global NPC space)
is a complete geodesic space which is globally nonpositively curved
in the sense of Alexandrov, i.e., Toponogov's triangle comparison
for nonpositive curvature holds for any geodesic triangle.
The class of Hadamard spaces, natural generalizations of
Cartan-Hadamard manifolds, includes all simply connected local NPC
spaces (see e. g. \cite{BuragoBuragoIvanov01}). When the target
space $(Y,d)$ is an Hadamard space, we have the following theorem.
\begin{theorem}\label{subharmonicity}
If $u\in W_{loc}^{1,2}(M,Y)$ is an $f$-harmonic map into an Hadamard
space $Y$, then for any $Q\in Y,$
\begin{equation}
-\int_M\<\nabla \eta(x),\nabla d(u(x), Q)\rangle e^{-f}dV_g\geq
0,\,\,\forall \,\,0\leq \eta\in \mathrm{Lip}_0(M),
\end{equation}
i.e., $d(u(x), Q)\in W_{loc}^{1,2}(M)$ is an $f$-subharmonic
function.
\end{theorem}

This theorem is a consequence of Jost \cite[Lemma 5]{Jost97}. The
subharmonicity of $d(u(\cdot), Q)$ for harmonic maps from an
admissible Riemannian polyhedron with simplexwise smooth Riemannian
metric to an Hadamard space was obtained by Eells-Fuglede
\cite[Lemma ~10.2]{EF}. Their argument essentially also works in our
setting. Using Remark \ref{fHarrem2}, Jost's lemma can be
reformulated in our setting as follows.
\begin{lemma}[{\cite[Lemma 5]{Jost97}}]\label{lemma2}
If $u\in W_{loc}^{1,2}(M, Y)$ is an $f$-harmonic map into an
Hadamard space $Y$, then for any $Q\in Y$ and $\eta\in
\mathrm{Lip}_0(M)$, $0\leq\eta\leq 1$,
\begin{equation}\label{JostIneq}
-\int_M\<\nabla \eta(x), \nabla d^2(u(x), Q)\rangle
e^{-f(x)}dV_g(x)\geq 2\int_M \eta(x)|\nabla u|^2(x)e^{-f(x)}dV_g(x).
\end{equation}
\end{lemma}
In fact, (\ref{JostIneq}) still holds for nonnegative functions
$\eta\in W^{1,2}(M)$ with compact support. (When $E^f(u)$ is finite,
(\ref{JostIneq}) even holds for $0\leq\eta\in W^{1,2}_0(M)$.) Now we
can prove Theorem \ref{subharmonicity} concerning the
$f$-subharmonicity of $d(u(\cdot),Q).$
\begin{proof}[Proof of Theorem \ref{subharmonicity}]
Denote by $\varphi(x):=\sqrt{x^2+\epsilon}$ for $\epsilon>0$.
For any $0\leq \eta\in \mathrm{Lip}_0(M),$ we choose a compactly
supported function
$$\eta_1(x):=\frac{\eta(x)}{2\varphi(d(u(x),Q))}\in W^{1,2}(M).$$
Then we calculate (we suppress the measure $e^{-f}dV_g$ in the
notations)
\begin{align*}
&-\int_M\left\langle\nabla\eta(x),\nabla\sqrt{d^2(u(x), Q)+\epsilon}\right\rangle =-\int_M\left\langle\nabla\eta(x),\frac{\nabla d^2(u(x),Q)}{2\varphi(d(u(x),Q))}\right\rangle\\
=&-\int_M\<\nabla\eta_1(x),\nabla
d^2(u(x),Q)\rangle-\int_M2\eta_1\frac{d(u(x),Q)\varphi'(d(u(x),Q))}{\varphi(d(u(x),Q))}|\nabla
d(u(x), Q)|^2.
\end{align*}
Note that
$$\frac{d(u(x),Q)\varphi'(d(u(x),Q))}{\varphi(d(u(x),Q))}=\frac{d^2(u(x),Q)}{d^2(u(x),
Q)+\epsilon}\leq 1,$$ and by (\ref{fHar3}), $$|\nabla d(u(x),
Q)|^2\leq |\nabla u(x)|^2,$$ we obtain
\begin{equation}
-\int_M\left\langle\nabla\eta(x), \nabla\sqrt{d^2(u(x),
Q)+\epsilon}\right\rangle\geq -\int_M\<\nabla\eta_1(x), \nabla
d^2(u(x),Q)\rangle-2\int_M\eta_1|\nabla u(x)|^2.
\end{equation}
Applying Lemma \ref{lemma2}, and letting $\epsilon\rightarrow 0$, we
complete the proof.
\end{proof}

Now we adopt the method of Li-Wang \cite{LiWang98}, a geometric
analysis method, to prove Kendall's theorem when the target is a
locally compact Hadamard space.
\begin{proof}[Proof of Theorem \ref{Kendall}]
By assumption, the space of bounded $f$-harmonic functions is of
dimension one. Then by the arguments of Grigor'yan
\cite{Grigoryan90}, every two $f$-massive subsets of $M$ have a
non-empty intersection. Here by a $f$-massive subset, we mean an open
proper subset of $\Omega\subset M$ on which there is a bounded,
nonnegative, nontrivial, $f$-subharmonic function $h$ such that
$h_{\mid\partial\Omega}=0$. Such function $h$ is called an
$f$-potential of the set $\Omega$.

Let $\hat{M}$ be the Stone-C\v{e}ch compactification of $M$. Then
every bounded continuous functions on $M$ can be continuously
extended to $\hat{M}$. Let $\Omega$ be an $f$-massive subset of $M$,
we then define the set $$S:=\bigcap_{\substack{h:
\text{$f$-potential}\\ \text{functions of
}\Omega}}\{\hat{x}\in\hat{M}\mid h(\hat{x})=\sup h\}.$$ By the
maximum principle for $f$-subharmonic functions, we know $S\subset
\hat{M}\setminus M$.

Then, by the same arguments as in \cite[Theorem~2.1]{LiWang98}, we
can prove $S\neq \emptyset$. Furthermore, for any bounded
$f$-subharmonic function $v$, we have
$S\subset\{\hat{x}\in\hat{M}\mid v(\hat{x})=\sup v\}$.

Let us take a point $Q_0\in \overline{u(M)}$. If $u(M)=\{Q_0\}$,
then we complete the proof. Otherwise, we have
$u(M)\setminus\{Q_0\}\neq \emptyset$. Since $u$ is an $f$-harmonic
map, by Theorem \ref{subharmonicity}, the function $h_1(x):=d(u(x),
Q_0)$ is an $f$-subharmonic function, which is bounded and
nonconstant. Hence $h_1$ attains its maximum at every point of $S$.
For a point $\hat{x}\in S$, there is a sequence $\{x_n\}$ in $M$
converging to $\hat{x}$ in $\hat{M}$. Note that $u$ has bounded
image. Thus by local compactness of the target $Y,$ there exists a
subsequence of $\{u(x_n)\}$ converging to $Q_1\in Y$. Now again, if
$u(M)=\{Q_1\}$, we complete the proof. Therefore, we can assume
$u(M)\setminus\{Q_1\}\neq\emptyset$. By Theorem
\ref{subharmonicity}, the function $h_2(x):=d(u(x), Q_1)$ is a
bounded $f$-subharmonic function. Thus $h_2$ achieves its maximum on
$S$, in particular at $\hat{x}$. That is
\begin{equation*}
\sup h_2(x)=h_2(\hat{x})=d(Q_1, Q_1)=0.
\end{equation*}
This contradicts our assumption. Therefore $u(M)=\{Q_1\}$ is a
constant map.
\end{proof}

\begin{remark}\label{remarkKuwae}
As pointed out to us by K. Kuwae, one can prove Kendall's theorem by
combining the methods of Li-Wang \cite{LiWang98} and Kuwae-Sturm
\cite{KuwaeSturm08} for harmonic maps into Hadamard spaces if the weak topology on the target (see \cite[Definition~2.7]{Jost94}) coincides with
the strong one, or equivalently any distance function $d(x,\cdot)$ on the target is weakly continuous for any $x\in
Y.$
\end{remark}

\

\section{Liouville type theorems}\label{section5}

In this section, we shall prove our main theorem. First of all, let
us review the classical classification theory of Riemannian
manifolds in the framework of weighted manifolds. For more details
we refer to \cite{GlasnerNakai72} and \cite{SarioNakai70}.

We recall some function spaces of $(M, g, e^{-f}dV_g)$. Let $D^f(M)$
be the set of Tonelli functions\footnote{A Tonelli function is a
continuous function with locally $L^2$-integrable weak derivatives.}
on $M$ with finite Dirichlet $f$-energy. The Royden algebra
$BD^f(M)$ is the set of bounded functions in $D^f(M)$. Under the
norm $\|u\|=\sup_M |u|+\sqrt{D^f(u)}$, $BD^f(M)$ becomes a Banach
algebra. For a sequence $\{u_n\}$ in $D^f(M)$, we say $u=C-\lim u_n$
if $u_n$ converges to $u$ uniformly on compact subsets and $u=B-\lim
u_n$ if in addition $\{u_n\}$ is uniformly bounded. We say
$u=D^f-\lim u_n$ if $\lim D^f(u_n-u)=0$. We also denote by
$u=CD^f-\lim u_n$ or $u=BD^f-\lim u_n$ to indicate two types of
convergence.

Let $C_0^{\infty}(M)$ be the set of smooth functions with compact
support and $D^f_0(M)$ be its closure under the $CD^f$-topology.  We
also denote by $HD^f(M)$ and $HBD^f(M)$ the sets of $f$-harmonic
functions in $D^f(M)$ and $BD^f(M)$ respectively.

\begin{proposition}\label{prop}
Let $(M, g, e^{-f}dV_g)$ be an $f$-parabolic weighted Riemannian
manifold. Then any $f$-subharmonic function with finite Dirichlet
$f$-energy is constant. In particular, any function in $HD^f(M)$ is
constant.
\end{proposition}
\begin{proof}
Let $u\in D^f(M)$ be  $f$-subharmonic. we may assume $u\geq 0$ since
$\max\{u,0\}$ is also $f$-subharmonic. Let $\{M_n\}$ be an
exhaustion of $M$ and take $w_k\in BD^f(M)$ with $w_k|_{M_0}=1,$
$w_k|_{M\backslash M_k}=0$ and $f$-harmonic in
$M_k\backslash\overline{M_0}$. It follows from the $f$-parabolicity
of $M$ that $BD^f-\lim w_k=1$. On the other hand, set $v_k\in
BD^f(M)$ with  $v_k|_{M_0}=u,$ $v_k|_{M\backslash  M_k}=0$ and
$f$-harmonic in $M_k\backslash\overline{M_0}$, one can verify that
$v=BD^f-\lim v_k$ exists. Set now $\tilde{u}=u-v$, and
$\tilde{u}_m=\min\{\tilde{u},m\}$. Then $\tilde{u}=D^f-\lim
\tilde{u}_m$. Since $\tilde{u}$ is nonnegative and $f$-subharmonic,
we can compute
\begin{eqnarray}\label{ineq}
0\geq -\int_{M_k\backslash  M_0} \tilde{u}_mw_k \Delta_f \tilde{u}
e^{-f}dV_g=\int_{M}  \<\nabla(\tilde{u}_mw_k),
\nabla\tilde{u}\rangle e^{-f}dV_g.
\end{eqnarray}
As $w_k\to 1$ in $D^f$-topology, we deduce from \eqref{ineq} by
letting $k\to\infty$ that $$\int_{M}  \<\nabla \tilde{u}_m,
\nabla\tilde{u}\rangle e^{-f}dV_g=0,$$ which yields
$D^f(\tilde{u})=0$ by letting $m\to\infty$. Since
$\tilde{u}|_{M_0}=0$, we see $u=v$. Finally,
$$D^f(u)=\int_M \< \nabla u,\nabla v\rangle e^{-f}dV_g=\lim_{k\to\infty} \int_M\< \nabla u,\nabla v_k\rangle e^{-f}dV_g\leq 0,$$ and hence $u$ is a
constant.
\end{proof}

The following are the weighted version of the Royden-Nakai
decomposition theorem and the Virtanen theorem. The proofs are almost
the same as the unweighted case. For the convenience of the readers,
we shall give proofs here.

\begin{theorem}[Royden-Nakai's decomposition theorem]\label{royden} Let $(M, g, e^{-f}dV_g)$ be a non-$f$-parabolic weighted Riemannian manifold. Then any function $u\in D^f(M)$ has a unique decomposition $u=h+g$, where $h\in HD^f(M)$ and $g\in D^f_0(M)$. Moreover, if $u$ is $f$-subharmonic, then $u\leq h$.
\end{theorem}

\begin{proof}
Let $u\in D^f(M)$. Assume first $u\geq 0$. Let $\{M_k\}$ be an
exhaustion of $M$. Take $h_k\in D^f(M)$ such that $h_k$ is $f$-harmonic in $M_k$ and $h_k|_{M\backslash
M_k}=u$.  Denote $g_k=u-h_k$. It follows from the maximum principle that
$h_k\geq 0$. One can check that \begin{eqnarray*}
D^f(u)&=&\int_M (|\nabla h_k|^2+|\nabla g_k|^2+2\left<\nabla h_k, \nabla g_k\right>)e^{-f}dV_g
\\&=& D^f(h_k)+D^f(g_k),
\end{eqnarray*}
where in the second equality we used integration by parts and the fact that $g_k|_{M\backslash 
M_k}=0$ and $h_k$ is $f$-harmonic in $M_k$.
Similarly we have for $m\leq k$
$$D^f(h_k-h_m)= D^f(h_k)-D^f(h_m).$$ Thus $\{h_k\}$
is  a $D^f$-Cauchy sequence, i.e., $D^f(h_k-h_m)$ is small enough when $m,k$ is large enough. 

Let $w_k\in BD^f(M)$ with
$w_k|_{M_0}=1,$ $w_k|_{M\backslash M_k}=0$ and harmonic in
$M_k\backslash\overline{M_0}$. It follows from the
non-$f$-parabolicity of $M$ that $w= BD^f-\lim w_k$ satisfies
$D^f(w)>0$.

We can compute
$$\int_M\< \nabla g_k,\nabla w_k\rangle e^{-f}dV_g=\int_{ M_k\backslash \overline M_0} \< \nabla g_k,\nabla w_k\rangle e^{-f}dV_g = \int_{\partial M_0}g_k\frac{\partial w_k}{\partial
\nu}e^{-f}dA_g, $$ where $\nu$ is the unit inward normal of
$\partial M_0.$ Since $w_k$ is $f$-harmonic in $M_k\backslash \overline M_0$, it follows from the Hopf lemma that $\frac{\partial w_k}{\partial
\nu}>0$ along $M_0$. 
It follows that
\begin{eqnarray*}
(\inf_{\partial M_0} h_k-\sup_{\partial M_0} u) \int_{\partial M_0} \frac{\partial w_k}{\partial
\nu}e^{-f}dA_g &\leq& \int_{\partial M_0} -g_k\frac{\partial w_k}{\partial
\nu}e^{-f}dA_g
\\&= &  -\int_M\< \nabla g_k,\nabla w_k\rangle e^{-f}dV_g
\\&\leq &\left[D^f(g_k)D^f(w_k)\right]^\frac12\leq \left[D^f(u)D^f(w_k)\right]^\frac12.
\end{eqnarray*}

 Combining this with the fact that $$ \int_{\partial
M_0}\frac{\partial w_k}{\partial \nu}e^{-f}dV_g=D^f(w_k),$$
we find
$$\inf_{M_0} h_k\leq \inf_{\partial M_0} h_k\leq \sup_{M_0} u+\left[\frac{D^f(u)}{D^f(w_k)}\right]^\frac12.$$
Since $w= BD^f-\lim w_k$ satisfies
$D^f(w)>0$, we see  $\inf_{M_0} h_k$ is bounded. Consequenly, by the Harnack inequality for $f$-harmonic functions, $\sup_{M_0} h_k$ is bounded. Hence there exists a subsequence of $h_k$, still denote by $h_k$, such that $\{h_k\}$ is a $C^f$-Cauchy sequence.

Together with the fact $\{h_k\}$ is a $D^f$-Cauchy sequence, we conclude that $h_k$ converges to some $h$ in $CD^f$-topology and $h\in HD^f(M)$. It is direct to check that $g_k$ converges to $g=u-h$ in $CD^f$-topology and
thus $g\in D_0^f(M)$.

Furthermore, if $u$ is $f$-subharmonic, from the construction
of $h_k$ we see $u-h_k$ is $f$-subharmonic and vanishes on $\partial M_k$ and in turn by maximum principle that $h\geq u$.

If $u$ is not nonnegative, we can run the same process for
$u^+=\max\{u,0\}$ and $u^-=-\min\{u,0\}$ as before and get the same
result.

The uniqueness follows from the fact that any $h\in HD^f(M)$ and
$g\in D^f_0(M)$ satisfy $\int_M \<\nabla h, \nabla g\rangle
e^{-f}dV_g=0$.
\end{proof}

\begin{theorem}[Virtanen's theorem]\label{virtanen}
For every $u\in HD^f(M)$ there exists a sequence $h_k\in HBD^f(M)$
such that $u=CD^f-\lim h_k$. In particular, $M$ admits no
nonconstant $f$-harmonic function on $M$ with finite Dirichlet
$f$-energy if and only if  $M$ admits no nonconstant bounded
$f$-harmonic function on $M$ with finite Dirichlet $f$-energy.
\end{theorem}

\begin{proof}
We may assume $M$ is non-$f$-parabolic, since otherwise, any $u\in
HD^f(M)$ is constant, due to Proposition \ref{prop}, whence the
statement is trivial. We may also assume $u\geq 0$, since otherwise
we do the same analysis on $u^+$ and $u^-$. Set for any $k\in
\mathbb{N}$, $u_k=\min\{u,k\}$. Then $u_k$ is $f$-superharmonic and
$u=D^f-\lim u_k$. By Royden-Nakai decomposition, $u_k=h_k+g_k$,
where $h_k\in HD^f(M)$ and $g_k\in D_0^f(M)$. Moreover, $g_k\geq 0$.
One can verify
$$D^f(u-u_k)=D^f(u-h_k)+D^f(g_k).$$
Hence $D^f(u-h_k)\to 0$ and $D^f(g_k)\to 0.$ Since $0\leq g_k\leq
u_k\leq u$ is bounded in any compact set of $M$, we conclude that
$g_k$ converges to some constant function $c$ in $CD^f$-topology. It
follows from the non-$f$-parabolicity of $M$ that $c=0$. Therefore
$h_k$ converges to $u$ in $CD^f$-topology.

The second assertion follows easily from this approximation.
\end{proof}

The following lemma was first proved by Cheng-Tam-Wan
\cite[Theorem~1.2]{CTW}.

\begin{lemma}\label{CTW}
Let $(M, g, e^{-f}dV_g)$ be a weighted Riemannian manifold. Then the
following two statements are equivalent:
\begin{itemize}
\item[(i)] any $u\in HD^f(M)$ is bounded;
\item[(ii)] any nonnegative $f$-subharmonic function on $M$ with finite Dirichlet $f$-energy is bounded.
\end{itemize}
\end{lemma}

\begin{proof}
(ii)$\Rightarrow $(i). This is quite simple by observing the fact
that if $u\in HD^f(M)$, then $\sqrt{u^2+1}$ is a nonnegative
$f$-subharmonic function on $M$ with finite Dirichlet $f$-energy.

(i)$\Rightarrow $(ii).  Assume $u$ is a nonnegative $f$-subharmonic
function on $M$ with finite Dirichlet $f$-energy. If $M$ is
$f$-parabolic, then the two statements are both true by virtue of
Proposition \ref{prop} and hence equivalent. If $M$ is
non-$f$-parabolic, then by Theorem \ref{royden}, $u=h+g$ for $h\in
HD^f(M)$ and $g\in D^f_0(M)$. Moreover, since $u$ is $f$-subharmoic,
we know $u\leq h$. By the assumption (i), $h$ is bounded. Thus $u$
is also bounded. This proves the lemma.
\end{proof}

Using Lemma \ref{CTW}, we can prove the main Theorem \ref{Cheng}.
\begin{proof}[Proof of Theorem \ref{Cheng}]
Let $u$ be an $f$-harmonic map from $M$ to $Y$ with finite
$f$-energy. It follows from Theorem \ref{subharmonicity} that the
function $v:M\to \R, \quad v(x)=\sqrt{d^2(u(x),Q)+1}$  is
subharmonic, where $Q\in Y$. Also, the finiteness of the $f$-energy
of $u$ implies the finiteness of the Dirichlet $f$-energy of $v$
(recall (\ref{fHar3})).  Using the assumption and the equivalence in
Lemma  \ref{CTW}, we know that any nonnegative $f$-subharmonic
function on $M$ with finite Dirichlet $f$-energy is bounded. Hence
$v$ is bounded, in turn, $u$ has bounded image. This proves the
theorem.
\end{proof}

For harmonic maps from $f$-parabolic weighted manifolds, we don't
need the local compactness assumption of the targets to obtain the
Liouville theorem.

\begin{corollary}
Let $(M,g, e^{-f}dV_g)$ be a complete noncompact $f$-parabolic
weighted Riemannian manifold and $(Y,d)$ be an Hadamard space. Then
any $f$-harmonic map from $M$ to $Y$ with
finite $f$-energy 
is a constant map.
\end{corollary}
\begin{proof}
Let $u$ be an $f$-harmonic map from $M$ to $Y$ with finite
$f$-energy. By Proposition \ref{prop} and Theorem \ref{Cheng}, the
image of $u$ is bounded. Hence for any $Q\in Y,$ the $f$-subharmonic
function $d(u(x),Q)$ is bounded. By the $f$-parabolicity of $M$ and
Proposition \ref{f-parabolic}, the function $d(u(x),Q)$ is constant
for any $Q\in Y.$ This yields that $u$ is a constant map. The
corollary follows.
\end{proof}

Combining Theorem \ref{Kendall} and Theorem \ref{Cheng}, we obtain
Theorem \ref{main thm} by the potential theory.
\begin{proof}[Proof of Theorem \ref{main thm}.]
By assumption,  any bounded $f$-harmonic function on $M$ is constant. By Theorem \ref{virtanen}, we know that
any $f$-harmonic function on $M$ with finite Dirichlet $f$-energy is
constant. Using Theorem \ref{Cheng}, we see that any $f$-harmonic
map from $M$ to $Y$ with finite $f$-energy must have bounded image.

On the other hand,  by Theorem \ref{Kendall}, we know that any
$f$-harmonic map from $M$ to $Y$ having bounded image is constant.
Hence any $f$-harmonic map from $M$ to $Y$ with finite $f$-energy
is a constant map. This proves the theorem.
\end{proof}

\begin{proof}[Proof of Theorem \ref{main thm1}.] By a theorem of Brighton \cite{Brighton13}, the weighted manifold $(M,g, e^{-f}dV_g)$
satisfying $Ric_f\geq 0$ admits no nonconstant bounded $f$-harmonic
functions. The assertion follows from Theorem \ref{main thm} immediately.

\end{proof}

\

\noindent\textbf{Acknowledgements}. The authors thank Prof. J\"urgen Jost for
numerous inspiring discussions on harmonic maps into singular
spaces. The authors thank Prof. Kazuhiro Kuwae for the discussions
on Kendall's theorem and suggestions on the conditions of the
targets (see Remark~\ref{remarkKuwae}). We would also like to thank the referee for his critical reading and useful suggestions.

\

\end{document}